\documentclass[11pt, a4paper, reqno, english]{amsart}

\usepackage{hyperref}
\usepackage{amssymb}
\usepackage[all]{xy}
\usepackage{tikz}

\renewcommand\AA{\mathbb{A}}
\newcommand\CC{\mathbb{C}}
\newcommand\FF{\mathbb{F}}
\newcommand\PP{\mathbb{P}}
\newcommand\ZZ{\mathbb{Z}}
\newcommand\NN{\mathbb{N}}
\newcommand\QQ{\mathbb{Q}}
\newcommand\RR{\mathbb{R}}
\newcommand\GG{\mathbb{G}}
\newcommand\K{\mathcal{K}}

\newcommand\GmZ{\GG_{\mathrm{m},\ZZ}}

\newcommand\Od{{\mathcal O}_K}
\newcommand\Odnz{{\Od^{\neq 0}}}
\newcommand\Nd{\NN_K}
\newcommand\wK{w_K}
\newcommand\dK{d_K}
\newcommand\Qbar{{\overline{\QQ}}}
\newcommand\Fbar{{\overline{\FF}}}
\renewcommand\tt{\mathbf{t}}
\newcommand\xx{\mathbf{x}}
\newcommand\yy{\mathbf{y}}
\newcommand\zz{\mathbf{z}}
\newcommand\dd{\mathbf{d}}
\newcommand\kk{\mathbf{k}}
\newcommand\tk{\smash{\widetilde{k}}}
\newcommand\tkk{\smash{\widetilde{\kk}}}
\newcommand\td{\smash{\widetilde{d}}}
\newcommand\tdd{\smash{\widetilde{\dd}}}
\newcommand\tmu{\smash{\widetilde{\mu}}}
\newcommand\rr{\mathbf{r}}
\newcommand\tr{\smash{\widetilde{r}}}
\newcommand\trr{\smash{\widetilde{\rr}}}
\newcommand\nn{\mathbf{n}}
\newcommand\tn{\smash{\widetilde{n}}}
\newcommand\tnn{\smash{\widetilde{\nn}}}
\newcommand\ddd{\mathrm{d}}

\newcommand\M{\mathcal{M}}
\newcommand\I{\mathcal{I}}
\newcommand\R{\mathcal{R}}

\newcommand{\Atwo}{{\mathbf A}_2}
\newcommand{\tS}{{\smash{\widetilde S}}}
\newcommand{\UZ}{{\mathcal U}}
\newcommand{\VZ}{{\mathcal V}}
\newcommand{\SZ}{{\mathcal S}}
\newcommand{\tSZ}{{\smash{\widetilde\SZ}}}

\DeclareMathOperator\Pic{Pic}

\DeclareMathOperator\Ci{Ci}
\DeclareMathOperator\rk{rk}
\DeclareMathOperator\Gal{Gal}

\newcommand\rto{\dashrightarrow}

\newcommand\gothp{\mathfrak{p}}
\newcommand\Norm{\mathfrak{N}}

\newtheorem*{theorem}{Theorem}
\newtheorem{lemma}{Lemma}
\theoremstyle{definition}
\newtheorem*{ack}{Acknowledgements}

\begin{document}

\title[Gaussian rational points on a singular cubic surface]{Gaussian
  rational points\\ on a singular cubic surface}

\author{Ulrich Derenthal}

\address{Mathematisches Institut, Ludwig-Maximilians-Universit\"at M\"unchen, 
  Theresienstr. 39, 80333 M\"unchen, Germany}

\email{ulrich.derenthal@mathematik.uni-muenchen.de}

\author{Felix Janda}

\address{Departement Mathematik, ETH Z\"urich, R\"amistr. 101, 8092 Z\"urich, Switzerland}

\email{felix.janda@math.ethz.ch}

\date{April 5, 2012}

\begin{abstract}
  Manin's conjecture predicts the asymptotic behavior of the number of
  rational points of bounded height on algebraic varieties. For toric
  varieties, it was proved by Batyrev and Tschinkel via height zeta
  functions and an application of the Poisson formula.  An alternative
  approach to Manin's conjecture via universal torsors was used so far
  mainly over the field $\QQ$ of rational numbers. In this note, we
  give a proof of Manin's conjecture over the Gaussian rational
  numbers $\QQ(i)$ and over other imaginary quadratic number fields
  with class number $1$ for the singular toric cubic surface defined
  by $x_0^3=x_1x_2x_3$.
\end{abstract}


%
%

\maketitle

\tableofcontents

\section{Introduction}

Let $S\subset \PP^3$ be the cubic surface defined over $\QQ$ by the
equation
\begin{equation*}
  x_0^3=x_1x_2x_3.
\end{equation*}
It is rational, toric and contains precisely three singularities and
three lines. Over any number field $K$, its set of $K$-rational points
is clearly infinite. Let $H$ be the Weil height on $S$, defined as
\begin{equation*}
  H(\xx)=\prod_{\nu \in M_K}\max_{j\in \{0, \dotsc, 3\}}\|x_j\|_\nu
\end{equation*}
where $\xx=(x_0:\dotso:x_3) \in S(K)$ with $x_0,\dotsc,x_3 \in K$, the
set of places of $K$ is denoted as $M_K$ and $\|\cdot\|_\nu$ is the
(suitably normalized; see Section~\ref{sec:peyre}) norm at the place
$\nu$. The total number of $K$-rational points of bounded height on $S$
is dominated by the number of easily countable points on the three
lines. Therefore, we restrict our attention to $K$-rational points in
the complement $U$ of the lines on $S$.

A much more general conjecture of Manin \cite{MR89m:11060} predicts in
case of $S$ that the number
\begin{equation*}
  N_{U,K,H}(B) = \#\{\xx \in U(K) \mid H(\xx) \le B\}
\end{equation*}
of $K$-rational points of bounded height outside the lines behaves
asymptotically as
\begin{equation*}
  N_{U,K,H}(B) \sim c_{S,K,H}B(\log B)^6,
\end{equation*}
as $B \to \infty$. A conjectural interpretation of the leading constant
$c_{S,K,H}>0$ was given by Peyre \cite{MR1340296} and refined by Batyrev and
Tschinkel \cite{MR1679843}.

Making use of the torus action on toric varieties to study the height
zeta functions and to apply the Poisson formula, Manin's conjecture
was proved for toric varieties over any number field by Batyrev and
Tschinkel \cite{MR1620682}; see \cite[\S 5.3]{MR1679843} for the
application of this result to our cubic surface $S$.

For varieties without such an action of an algebraic group, an
alternative approach using \emph{universal torsors} was suggested by
Salberger \cite{MR1679841}. He gave a second proof of Manin's
conjecture over $\QQ$ in the case of split toric varieties; see
\cite[Example~11.50]{MR1679841} for its application to $S$.

For the singular cubic surface $S$ as above, Manin's conjecture over
$\QQ$ was also proved directly by Fouvry \cite{MR2000b:11075},
Heath-Brown and Moroz \cite{MR2000f:11080}, de la Bret\`eche
\cite{MR2000b:11074}, de la Bret\`eche and Swinnerton-Dyer
\cite{MR2430199} and Bhowmik, Essouabri, Lichtin \cite{MR2367957},
using elementary or classical analytic number theoretic techniques and
a parameterization of rational points closely related to universal
torsors in some cases.

The basic example of a universal torsor applied to point counting is
the following: To estimate the number
\begin{equation*}
  N_{\PP^n,\QQ,H}(B) = \#\{\xx \in \PP^n(\QQ) \mid H(\xx) \le B\}
\end{equation*}
of rational points of bounded height in $n$-dimensional projective
space $\PP^n$, the natural first step is the observation that any such
$\xx$ is represented uniquely up to sign by an $(n+1)$-tuple of
coprime integers $(x_0, \dotsc, x_n)$ subject to the condition
$\max\{|x_0|, \dotsc, |x_n|\} \le B$. Geometrically, this corresponds
to the fact that the open subset $\AA^{n+1} \setminus \{0\}$ of
$(n+1)$-dimensional affine space is a universal torsor over $\PP^n$.

Based on this, Schanuel \cite{MR0162787}, \cite{MR557080} proved
Manin's conjecture for projective spaces over arbitrary number
fields. Over number fields other than $\QQ$, no other proof of Manin's
conjecture via universal torsors is known to us.

The purpose of this note is to begin the generalization of universal torsor
techniques from $\QQ$ to more general number fields.  A first candidate is
Manin's conjecture for the toric cubic surface $S$ over the field $\QQ(i)$ of
Gaussian rational numbers because its class number is $1$ and its ring of
integers contains only finitely many units. It turns out that it is not too
hard to generalize from $\QQ(i)$ to the following setting:

\begin{theorem}\label{thm:main}
  Let $K$ be an imaginary quadratic number field whose class number is
  $1$, let $\wK$ be the number of units in its ring of integers $\Od$,
  and let $\dK$ be the square root of the absolute value of its
  discriminant (cf.\ Table~\ref{tab:numberfields}).  Let $S \subset
  \PP^3$ be the cubic surface defined by $x_0^3=x_1x_2x_3$. Let $U$ be
  the complement of the three lines on $S$. Then
  \begin{equation*}
    N_{U,K,H}(B) \sim c_{S,K,H} B(\log B)^6+O(B(\log B)^5)
  \end{equation*}
  as $B \to \infty$, with
  \begin{equation*}
    c_{S,K,H}=\frac{2^7\pi^9}{6!\wK^7\dK^9}\prod_p\left(1-\frac{1}{\|p\|_\infty}\right)^7
    \left(1+\frac{7}{\|p\|_\infty}+\frac{1}{\|p\|_\infty^2}\right)
  \end{equation*}
  where the product runs over all primes in $\Od$ up to units.
\end{theorem}

\begin{table}[h]
  \centering
  \begin{equation*}
    \begin{array}[h]{|c||c|c|c|c|c|c|c|c|c|}
      \hline
      n & -1 & -2 & -3 & -7 & -11 & -19 & -43 & -67 & -163 \\
      \hline\hline
      \wK & 4 & 2 & 6 & 2 & 2 & 2 & 2 & 2 & 2 \\
      \hline
      \dK & 2 & 2\sqrt 2 & \sqrt 3 & \sqrt 7 & \sqrt{11} & \sqrt{19} & \sqrt{43} & \sqrt{67} & \sqrt{163}\\
      \hline
    \end{array}
  \end{equation*}
  \caption{$K=\QQ(\sqrt n)$ with class number $h_K=1$.}
  \label{tab:numberfields}
\end{table}

We will see in Section~\ref{sec:peyre} that this result agrees with
the conjectures of Manin, Peyre, Batyrev and Tschinkel.

Recently, Frei \cite{arXiv:1204.0383} generalized our work to
arbitrary number fields, removing our restriction to class number $1$
and finite groups of units in the ring of integers.

\begin{ack}
  The authors are grateful to Tim Browning and the referee for helpful
  remarks. The first named author was supported by grant 200021\_124737/1 of
  the Schweizer Nationalfonds and by grant DE~1646/2-1 of the Deutsche
  Forschungsgemeinschaft.
\end{ack}

\section{Geometry}\label{sec:geometry}

In this section, we collect some facts on the geometry of our singular
cubic surface $S$. The construction of its minimal desingularization
as a blow-up of the projective plane in six points will be used in
Section~\ref{sec:passage} to construct a parameterization of the
$K$-rational points by integral points on a universal torsor.

Let $\SZ$ be the model of $S$ over $\ZZ$ defined by the equation
$x_0^3=x_1x_2x_3$ in $\PP^3_\ZZ$. We will consider the minimal
desingularization $\tSZ$ of $\SZ$, which is obtained from $\PP^2_\ZZ$
by a sequence of six blow-ups of points. All statements below will be
true not only over $\ZZ$ but, suitably rephrased, over any field.  In
what follows, all statements involving variables $j,k,l$ are meant to
hold for all
\begin{equation*}
  (j,k,l) \in \{(1,2,3),(2,3,1),(3,1,2)\}.
\end{equation*}

The surface $\SZ$ is singular, with
precisely three singularities
\begin{equation*}
  p_1=(0:1:0:0),\quad p_2=(0:0:1:0), \quad p_3=(0:0:0:1).
\end{equation*}
They are rational double points of type $\Atwo$ in the
$\mathbf{ADE}$-classification. It contains precisely three lines
$\ell_j=\{x_0=x_j=0\}$ through $p_k$ and $p_l$.

The surface $\SZ$ is toric. Indeed, an action of a two-dimensional torus
on $\SZ$ is given by
\begin{equation*}
  \GmZ^2 \times \SZ \to \SZ,\quad (\tt,\xx) \mapsto \tt
  \cdot \xx = (x_0:t_1x_1:t_2x_2:(t_1t_2)^{-1}x_3),
\end{equation*}
giving an isomorphism from $\GmZ^2$ to the open dense orbit
\begin{equation*}
  \UZ = \SZ \setminus(\ell_1\cup \ell_2 \cup \ell_3) = \{\xx \in \SZ \mid
  x_0x_1x_2x_3 \ne 0\}
\end{equation*}
of $(1:1:1:1)$, say. The corresponding fan can
be found in Figure~\ref{fig:fans}.

\begin{figure}[ht]
  \begin{tikzpicture}
    \draw[step=1cm,gray,thin,dotted] (-1,-2) grid (2,1);
    \begin{scope}[thick]
      \draw(0,0) -- (-1,1) node[anchor=south] {$\ell_2$};
      \draw(0,0) -- (2,1) node[anchor=west] {$\ell_1$};
      \draw(0,0) -- (-1,-2) node[anchor=north] {$\ell_3$};
    \end{scope}
  \end{tikzpicture}
  \begin{tikzpicture}
    \draw[step=1cm,gray,thin,dotted] (-1,-2) grid (2,1);
    \begin{scope}[thick]
      \draw(0,0) -- (1,0) node[anchor=west] {$E_{3,1}$};
      \draw(0,0) -- (0,1) node[anchor=south] {$E_{1,2}$};
      \draw(0,0) -- (-1,-1) node[anchor=east] {$E_{2,3}$};
    \end{scope}
    \begin{scope}[thick] 
      \draw(0,0) -- (-1,0) node[anchor=east] {$E_{3,2}$};
      \draw(0,0) -- (0,-1) node[anchor=north] {$E_{1,3}$};
      \draw(0,0) -- (1,1) node[anchor=south] {$E_{2,1}$};
    \end{scope}
    \begin{scope}[thick]
      \draw(0,0) -- (-1,1) node[anchor=south] {$E_2$};
      \draw(0,0) -- (2,1) node[anchor=west] {$E_1$};
      \draw(0,0) -- (-1,-2) node[anchor=north] {$E_3$};
    \end{scope}
  \end{tikzpicture}
  \begin{tikzpicture}
    \draw[step=1cm,gray,thin,dotted] (-1,-1) grid (1,1);
    \begin{scope}[thick]
      \draw(0,0) -- (1,0) node[anchor=west] {$E_{3,1}$};
      \draw(0,0) -- (0,1) node[anchor=south] {$E_{1,2}$};
      \draw(0,0) -- (-1,-1) node[anchor=north] {$E_{2,3}$};
    \end{scope}
    \begin{scope}[thick] 
      \draw(0,0) -- (-1,0) node[anchor=east] {$E_{3,2}$};
      \draw(0,0) -- (0,-1) node[anchor=north] {$E_{1,3}$};
      \draw(0,0) -- (1,1) node[anchor=south] {$E_{2,1}$};
    \end{scope}
  \end{tikzpicture}
  \begin{tikzpicture}
    \draw[step=1cm,gray,thin,dotted] (-1,-1) grid (1,1);
    \begin{scope}[thick]
      \draw(0,0) -- (1,0) node[anchor=west] {$E_{3,1}$};
      \draw(0,0) -- (0,1) node[anchor=south] {$E_{1,2}$};
      \draw(0,0) -- (-1,-1) node[anchor=north] {$E_{2,3}$};
    \end{scope}
  \end{tikzpicture}
  \caption{Fans of $\SZ, \tSZ, \tSZ_1, \PP^2_\ZZ$, respectively.}
  \label{fig:fans}
\end{figure}

Resolving the singularity $p_j$ gives two exceptional divisors
$E_{k,l}$ (meeting the strict transform $E_l$ of $\ell_l$) and
$E_{l,k}$ (meeting the strict transform $E_k$ of $\ell_k$). We obtain
the minimal desingularization $\pi: \tSZ \to \SZ$, where the Picard
group of $\tSZ$ is free of rank $7$. The six curves
$E_{j,k},E_{k,j}$ are $(-2)$-curves (rational curves with
self-intersection number $-2$), while the three transforms $E_j$ of
the lines are $(-1)$-curves (rational curves with self-intersection
number $-1$). There are no other negative curves (rational curves with
negative self-intersection number) of $\tSZ$. The negative curves
correspond precisely to the rays in the fan of $\tSZ$ in
Figure~\ref{fig:fans}.

The surface $\SZ$ is rational, via the birational map
\begin{equation*}
  \phi: \SZ \rto \PP^2_\ZZ,\ 
  \xx \mapsto (x_0^2:x_0x_1:x_1x_2)=(x_2x_3:x_0^2:x_0x_2)=(x_0x_3:x_1x_3:x_0^2),
\end{equation*}
(where the three expressions coincide where they are defined), with
inverse
\begin{equation*}
  \psi: \PP^2_\ZZ \rto \SZ, \quad 
  \zz \mapsto (z_0z_1z_2:z_1^2z_2:z_2^2z_0:z_0^2z_1).
\end{equation*}
Indeed, $\phi$ and $\psi$ restrict to isomorphisms between the open
subsets $\UZ \subset \SZ$ and $\VZ=\{\zz \in \PP^2_\ZZ \mid z_0z_1z_2
\neq 0\} \subset \PP^2_\ZZ$.

Then the following diagram commutes, where $\pi_0: \tSZ \to \PP^2_\ZZ$
is the blow-up of $\PP^2_\ZZ$ in six points in \emph{almost general
  position} \cite{MR579026}.
\begin{equation*}
  \xymatrix{\tSZ \ar@{->}[d]_\pi \ar@{->}[dr]^{\pi_0}& \\
    \SZ\ar@{-->}[r]_\phi & \PP^2_\ZZ}
\end{equation*}
More precisely, $\pi_0$ maps
\begin{itemize}
\item $E_1, E_{2,1}$ to $(1:0:0)$,
\item $E_2, E_{3,2}$ to $(0:1:0)$,
\item $E_3, E_{1,3}$ to $(0:0:1)$,
\item $E_{2,3}, E_{3,1}, E_{1,2}$ to $\{z_0=0\},
  \{z_1=0\}, \{z_2=0\}$, respectively.
\end{itemize}

Conversely, using the same symbol for divisors on $\tSZ$ and their
projections and strict transforms on $\PP^2_\ZZ$ and the intermediate
$\tSZ_1$ (for example, on $\PP^2_\ZZ$, we have $E_{2,3}=\{z_0=0\}$,
$E_{3,1}=\{z_1=0\}$ and $E_{1,2}=\{z_2=0\}$) we obtain
\begin{equation*}
  \pi_0:\tSZ \xrightarrow{\pi_2}  \tSZ_1 \xrightarrow{\pi_1} \PP^2_\ZZ,
\end{equation*}
where $\tSZ_1$ is a smooth sextic del Pezzo surface, by
\begin{itemize}
\item blowing up the three points $E_{j,k} \cap E_{k,l} \in \PP^2_\ZZ$
  with exceptional divisors $E_{l,k}$, respectively, to obtain $\pi_1:
  \tSZ_1 \to \PP^2_\ZZ$;
\item blowing up the three points $E_{k,j} \cap E_{l,j} \in \tSZ_1$
  with exceptional divisors $E_j$, respectively, to obtain $\pi_2:
  \tSZ \to \tSZ_1$.
\end{itemize}
This gives $\tSZ$ with three $(-1)$-curves $E_j$ and six
$(-2)$-curves $E_{j,k},E_{k,j}$. Contracting the
$(-2)$-curves via the anticanonical map gives $\pi: \tSZ \to \SZ
\subset \PP^3_\ZZ$.

\section{The leading constant}\label{sec:peyre}

For a smooth Fano variety defined over a number field $K$, Peyre
\cite[Conjecture~2.3.1]{MR1340296} gave a conjectural interpretation
of the leading constant in Manin's conjecture. This was generalized to
Fano varieties with at worst canonical singularities by Batyrev and
Tschinkel \cite[\S 3.4 Step~4]{MR1679843}. We will
see that our theorem agrees with this prediction.

We start by collecting all number theoretic notation we need for this
section.  Let $(r_K,s_K)$ be the number of real resp.\ pairs of
complex embeddings of $K$, and let $q_K=r_K+s_K-1$. Let $\Od$ be its
ring of integers.  Let $\Odnz = \Od \setminus \{0\}$. Let $\wK$ be
number of roots of unity in $\Od$, and let $R_K$ be the regulator of
$K$. Let $\dK$ denote the square root of the absolute value of the
discriminant of $K$.

The set $M_K$ of places of $K$ consists of the archimedian places
$M_{K,\infty}$ and the non-archimedian places $M_{K,f}$. For $\nu \in
M_K$, let $K_\nu$ be the completion of $K$ at $\nu$ and, for $\nu \in
M_{K,f}$, let $\FF_\nu$ be the residue field. For any $\nu \in
M_{K,f}$, we define a norm by $\|x\|_\nu = |N_{K_\nu/\QQ_p}(x)|_p$ for
all $x \in K_\nu$, where $p$ the characteristic of $\FF_\nu$ and
$|\cdot|_p$ is the usual norm on $\QQ_p$. For any $\nu \in
M_{K,\infty}$ corresponding to a real embedding $\sigma: K \to \RR$,
we define $\|x\|_\nu = |\sigma(x)|$ for all $x \in K_\nu$, where
$|\cdot|$ is the usual absolute value on $\RR$. For any $\nu \in
M_{K,\infty}$ corresponding to a pair of complex embeddings
$\sigma,\sigma'$, we define $\|x\|_\nu = |\sigma(x)|^2$, where
$|\cdot|$ is the usual absolute value on $\CC$.

We compute the expected asymptotic behavior of $N_{U,K,H}(B)$ with respect to
the very ample anticanonical metrized sheaf $-\K_S = (-K_S, \|\cdot\|_\nu)$
\cite[Definition~3.1.3]{MR1679843}, where the family of $\nu$-adic metrics
corresponds to our anticanonical height function $H$.

We use the minimal desingularization $\pi: \tS \to S$ and its integral model
$\pi: \tSZ \to \SZ$ constructed in Section~\ref{sec:geometry}. As the
$\Atwo$-singularities on $S$ are rational double points, we have $\pi^*(K_S) =
K_\tS$. Therefore, the $-\K_S$-index \cite[Definition~2.2.4]{MR1679843} of $S$
is $1$, and the $-\K_S$-rank \cite[Definition~2.3.11]{MR1679843} of $S$ is
$\rk\Pic(\tS)=7$. Hence the expected asymptotic formula according to
Manin's conjecture is (in the notation of \cite[\S 3.4 Step~4]{MR1679843})
\begin{equation*}
  N_{U,K,H}(B)=\frac{\gamma_{-\K_S}(U)}{6!}\delta_{-\K_S}(U)\tau_{-\K_S}(U) B
  (\log B)^6(1+o(1)).
\end{equation*}

The cohomological factor of the expected leading constant is
\begin{equation*}
  \delta_{-\K_S}(U)=\#H^1(\Gal(\Qbar/K),\Pic(\tS_\Qbar)) = 1,
\end{equation*}
as $\Gal(\Qbar/K)$ acts trivially on $\Pic(\tS_\Qbar)$ since $\tS$ is split
over $K$.

The factor $\gamma_{-\K_S}(U)/6!$ \cite[Definition~2.3.16]{MR1679843} is
simply $\alpha(\tS)$ as in \cite[D\'efinition~2.4]{MR1340296}. By
\cite[Theorem~1.3]{MR2377367}, we have
\begin{equation*}
  \frac{\gamma_{-\K_S}(U)}{6!} = \alpha(\tS) = 
\frac{\alpha(S_0)}{\#W(3\Atwo)} = \frac{1}{120\cdot (3!)^3} = \frac{1}{36\cdot 6!},
\end{equation*}
where $S_0$ is a smooth cubic surface with $\alpha(S_0)=1/120$ by
\cite[Theorem~4]{MR2318651} and $W(3\Atwo)$ is the Weyl group of the root
system $3\Atwo$ associated to the singularities of $S$.

Next, we compute the Tamagawa number
\begin{equation*}
\tau_{-\K_S}(U) = \lim_{s\to 1}(s-1)^7
L(s,\Pic(\tS_\Qbar))\int_{\tS(\AA_K)}\omega_{-\K_S}
\end{equation*}
\cite[Definition~3.3.10]{MR1679843}, where the set $\tS(\AA_K)$ of adelic
points coincides with the closure of $\tS(K)$ in it since $\tS$ satisfies weak
approximation. Here, we have used that every non-archimedian valuation $\nu$
is a good valuation in the sense of \cite[Definition~3.3.5]{MR1679843} since
the reduction of the model $\tSZ$ of $\tS$ at any finite place of $K$ is a smooth
projective variety.

Since $\tSZ$ is split, the Frobenius morphism associated to every
non-archime\-dian place $\nu$ corresponding to a prime ideal $\gothp$
acts trivially on $\Pic(\tSZ_{\Fbar_\nu})$ of rank $7$. Therefore,
$L_\nu(s,\Pic(\tS_\Qbar))=(1-\Norm\gothp^{-s})^{-7}$ (cf.\ \cite[\S
2.2.3]{MR1340296}), and $L(s,\Pic(\tS_\Qbar)) = \prod_{\nu \in
  M_{K,f}} L_\nu(s,\Pic(\tS_\Qbar)) = \zeta_K(s)^7$. So
\begin{equation*}
  \lim_{s\to 1}(s-1)^7 L(s,\Pic(\tS_\Qbar))=\lim_{s\to
    1}(s-1)^7\zeta_K(s)^7  =\left(\frac{2^{r_K}(2\pi)^{s_K}h_KR_K}{\wK\dK}\right)^7
\end{equation*}
by the analytic class number formula.

Furthermore, by \cite[Definition~3.3.9]{MR1679843},
\begin{equation*}
  \int_{\tS(\AA_K)} \omega_{-\K_S} = \dK^{-\dim(\tS)} \prod_{\nu\in
    M_K}\lambda_\nu^{-1}d_\nu(U)
\end{equation*}
where $\lambda_\nu = L_\nu(1,\Pic(\tS_\Qbar))$ for all (good) non-archimedian
places and $\lambda_\nu = 1$ for the archimedian places. It remains to compute
the local densities $d_\nu(U)$ defined in \cite[Remark~3.3.2]{MR1679843}.

We compute the archimedian densities on the open subset $U = \{x_0 \ne 0\}$ of
$S$, defined by the cubic equation $f(x_0, \dotsc, x_3)=x_0^3-x_1x_2x_3$, as
\begin{equation*}
  d_\nu(U) = \int_{S(K_\nu)} \omega_{-\K_S}(g)=
  \begin{cases}
    36, &\text{$\nu$ real,}\\
    36\pi^2, &\text{$\nu$ complex.}
  \end{cases}
\end{equation*}
Indeed, we apply \cite[Lemme~5.4.4]{MR1340296} and see via the birational
morphism $\rho: U \to \AA^2$ defined by $(x_0:x_1:x_2:x_3) \mapsto (x_1/x_0,
x_2/x_0)$ that
\begin{equation*}
  d_\nu(U) = \int_{(K_\nu^\times)^2}
  \frac{1}{\max\{1,\|y_1\|_\nu,\|y_2\|_\nu,\|(y_1y_2)^{-1}\|_\nu\}\cdot
  \|-y_1y_2\|_\nu} \ddd y_{1,\nu} \ddd y_{2,\nu}.
\end{equation*}
A straightforward computation gives the values above. For complex $\nu$, the
Haar measure $\ddd y_{i,\nu}$ on $K_\nu$ is normalized as \emph{twice} the
usual Lebesgue measure obtained from regarding as $K_\nu \cong \CC$ as
$\RR^2$, as in \cite[\S 1.1]{MR1340296}. For real $\nu$, the value $d_\nu(S) =
36$ can also be found in \cite[\S 5.3]{MR1679843}.

As every non-archimedian place $\nu$ corresponding to a prime ideal $\gothp$
is good, we can apply \cite[Theorem~3.3.7]{MR1679843} to compute
\begin{equation*}
  d_\nu(U) = \frac{\#\tS(\FF_\gothp)}{\Norm\gothp^2} = 1+\frac{7}{\Norm\gothp}+\frac{1}{\Norm\gothp^2},
\end{equation*}
since the norm $\Norm\gothp$ is the cardinality of the
residue field $\FF_\nu$ of $K_\nu$.  Indeed, for any finite field $\FF_q$, the
surface $\tSZ_{\FF_q}$ is the blow-up of $\PP^2_{\FF_q}$ in six
$\FF_q$-rational points, and any such blow-up replaces one
$\FF_q$-rational point by a rational curve containing $q+1$ points
over $\FF_q$. Since $\#\PP^2_\ZZ(\FF_q)=q^2+q+1$, we obtain the
result. See also \cite[Lemma~2.3]{MR2769338}.

In total, the expected leading constant is
\begin{equation*}
  \frac{9^{q_K}}{4\cdot 6!} \left(\frac{2^{r_K}(2\pi)^{s_K}}{\dK}\right)^9
  \left(\frac{h_KR_K}{\wK}\right)^7 \prod_\gothp
  \left(1-\frac{1}{\Norm\gothp}\right)^7\left(1+\frac{7}{\Norm\gothp}+\frac{1}{\Norm\gothp^2}\right).
\end{equation*}

For imaginary quadratic number fields $K$ with class number $h_K=1$,
we have $(r_K,s_K)=(0,1)$, so $q_K=0$.  Since the number $\wK$ of
units in $\Od$ is finite, its regulator $R_K$ is $1$. We denote the
archimedian place as $\nu=\infty$. Let $\Nd$ be a fundamental domain
for $\Odnz$ modulo the action of the units. We identify each prime
ideal $\gothp$ with its unique generator $p \in \Nd$, with
$\Norm\gothp = \|p\|_\infty$. We see that the expected leading
constant of \cite{MR1679843} coincides with $c_{S,K,H}$ in our main
theorem.

\section{Passage to a universal torsor}\label{sec:passage}

We follow the strategy of \cite{MR2290499}. This leads to a
parameterization of rational points on $S$ by integral points in
$\AA^9$ that is similar to the one used in \cite{MR2000f:11080}, but
with a different set of coprimality conditions. We could construct
coprimality conditions as in \cite{MR2000f:11080}, but we believe our
conditions are more closely connected to the geometry of $\tS$ and
easier to work with. We note that our coprimality conditions are
analogous to the ones obtained by Salberger \cite[11.5]{MR1679841} for
toric varieties over $\QQ$.

In the following, any
statement involving $j,k,l$ is meant to hold for all
\begin{equation*}
  (j,k,l) \in \{(1,2,3),(2,3,1),(3,1,2)\}.
\end{equation*}

A parameterization of $K$-rational points on $U \subset S$ is
obtained via the map $\psi$ defined in Section~\ref{sec:geometry}.
The isomorphism $\psi_{|V}: V \to U$ induces a map
\begin{equation*}
  \Psi_0:(\Odnz)^3 \to S(K), \quad 
  \yy \mapsto \Psi_0(\yy) = 
  (\Psi_0(\yy)_0: \dotso : \Psi_0(\yy)_3)
\end{equation*}
where $\yy=(y_{2,3},y_{3,1},y_{1,2})$ and
\begin{equation*}
  \Psi_0(\yy)_0=y_{1,2}y_{3,1}y_{2,3}, \ 
  \Psi_0(\yy)_j = y_{j,k}y_{l,j}^2.
\end{equation*}
This induces a $\wK$-to-$1$ map from
\begin{equation*}
  \{(y_{2,3},y_{3,1},y_{1,2}) \in (\Odnz)^3 \mid 
  H(\Psi_0(\yy)) \le B,\ \gcd(y_{1,2},y_{2,3},y_{3,1})=1\}
\end{equation*}
to
\begin{equation*}
  N_0(B) = \{\xx \in U(K) \mid H(\xx) \le B\}.
\end{equation*}

However, this parameterization is not good enough to start counting
integral elements in a region in $\Od^3$ because the height condition
is not as easy as one might hope since $\gcd(y_{1,2}y_{2,3}y_{3,1},
y_{1,2}y_{3,1}^2, y_{2,3}y_{1,2}^2, y_{3,1}y_{2,3}^2)$ (taken here and
always in $\Od$) may be non-trivial even if
$\gcd(y_{1,2},y_{2,3},y_{3,1})=1$.

Motivated by the construction of $\tSZ$ as the blow-up of $\PP^2_\ZZ$ in
intersection points of certain divisors, we modify this as follows.

In the first step, let $y_{l,k} = \gcd(y_{j,k},y_{k,l})$. Write
$y_{j,k}=y_{j,k}'y_{k,j}y_{l,k}$. Then
$\gcd(y_{j,k}',y_{k,l}')=\gcd(y_{l,k},y_{k,j})=\gcd(y_{k,j},y_{k,l}')=1$.
Now we drop the $'$ again for notational simplicity. We obtain a map
\begin{equation*}
  \Psi_1: (\Odnz)^6 \to S(K), \quad 
  \yy \mapsto \Psi_1(\yy)=(\Psi_1(\yy)_0: \dotso : \Psi_1(\yy)_3),
\end{equation*}
where $\yy=(y_{1,2},y_{2,1},y_{1,3},y_{3,1},y_{2,3},y_{3,2})$ and
\begin{equation*}
  \Psi_1(\yy)_0=y_{1,2}y_{2,1}y_{1,3}y_{3,1}y_{2,3}y_{3,2}, \quad 
  \Psi_1(\yy)_j = y_{j,k}y_{j,l}y_{k,j}^2y_{l,j}^2.
\end{equation*}

We note that the coprimality conditions can be expressed as follows:
For $(u,v) \in \{(1,2),(2,1),(1,3),(3,1),(2,3),(3,2)\}$, we have
$\gcd(y_u,y_v)=1$ if and only if the divisors $E_u$ and
$E_v$ do not intersect on $\tSZ_1$, which holds if and only if the
corresponding rays in the fan of $\tSZ_1$ (Figure~\ref{fig:fans}) are
not neighbors.

Since the $y_{k,j}$ are unique up to
units in $\Od$, the map $\Psi_1$ induces a $\wK^4$-to-$1$ map from
\begin{equation*}
  \{\yy \in (\Odnz)^6 \mid H(\Psi_1(\yy)) \le B, 
  \text{coprimality as in the fan of $\tSZ_1$ in Figure~\ref{fig:fans}}\}
\end{equation*}
to $N_0(B)$.

In the second step, let $y_j=\gcd(y_{k,j},y_{l,j})$. As before, we
obtain a map
\begin{equation*}
  \Psi_2:(\Odnz)^9 \to S(K),\quad 
  \yy \mapsto \Psi_2(\yy)=(\Psi_2(\yy)_0: \dotso: \Psi_2(\yy)_3),
\end{equation*}
where $\yy=(y_1,y_2,y_3,y_{1,2},y_{2,1},y_{1,3},y_{3,1},y_{2,3},y_{3,2})$ and
\begin{equation*}
  \Psi_2(\yy)_0=y_1y_2y_3y_{1,2}y_{2,1}y_{1,3}y_{3,1}y_{2,3}y_{3,2}, \quad 
  \Psi_2(\yy)_j = y_j^3y_{j,k}y_{j,l}y_{k,j}^2y_{l,j}^2.
\end{equation*}
This induces a $\wK^7$-to-$1$ map from
\begin{equation*}
  \{\yy \in (\Odnz)^9 \mid H(\Psi_2(\yy)) \le B, 
  \text{coprimality as in the fan of $\tSZ$ in Figure~\ref{fig:fans}}\}
\end{equation*}
to $N_0(B)$.

Now we note that
\begin{equation*}
H(\Psi_2(\yy))=\max\{\|\Psi_2(\yy)_1\|_\infty,\|\Psi_2(\yy)_2\|_\infty,\|\Psi_2(\yy)_3)\|_\infty\}
\end{equation*}
because the coprimality conditions imply that $\Psi_2(\yy)_0, \dotsc,
\Psi_2(\yy)_3$ are coprime for any $\yy$ satisfying the coprimality
conditions, and the archimedian norm of $\Psi_2(\yy)_0$ cannot be
larger than all other three. Indeed, the second observation follows
from $\Psi_2(\yy)_0^3=\Psi_2(\yy)_1\Psi_2(\yy)_2\Psi_2(\yy)_3$. For
the first observation, we note that any prime may divide at most two
variables whose corresponding rays in Figure~\ref{fig:fans} are
neighbors, and one checks that for each such pair of variables, there
is one monomial in which these variables do not occur.

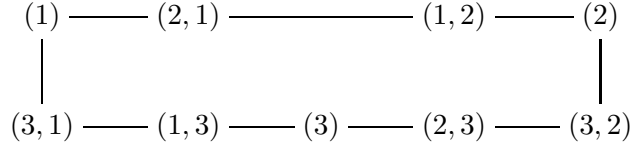
\begin{figure}[ht]
  \centering
  \begin{equation*}
    \xymatrix{(1) \ar@{-}[r]\ar@{-}[d] & (2,1) \ar@{-}[rr] & & (1,2) \ar@{-}[r] & (2) \\
      (3,1) \ar@{-}[r] & (1,3) \ar@{-}[r] & (3) \ar@{-}[r] & (2,3) \ar@{-}[r] & (3,2)\ar@{-}[u]
    }
  \end{equation*}
  \caption{Graph $G=(V,E)$ encoding coprimality conditions.}\label{fig:dynkin}
\end{figure}

We reformulate the coprimality conditions as follows, using the graph
$G=(V,E)$ with nine vertices $V=\{(1),(2,1),\dotsc\}$ and nine edges
$E=\{\{(1),(2,1)\},\{(2,1),(1,2)\}, \dotsc\}$ in
Figure~\ref{fig:dynkin}.

\begin{lemma}\label{lem:passage}
  Let $G=(V,E)$ be the graph in Figure \ref{fig:dynkin}. Let $E'$ be
  the set all pairs $\{u,v\}$ of vertices $u,v \in V$ which are not
  adjacent in the graph.
  
  We have
  \begin{equation*}
    N_{U,K,H}(B) = \frac{1}{\wK^7} \sum_{\substack{\yy \in (\Odnz)^V\cap M(B) \\
        \text{$\gcd(y_u,y_v)=1$ for all $\{u,v\}\in E'$}}} 1,
  \end{equation*}
  where $M(B)$ is the set of all $\yy\in\CC^V$ with
  \begin{equation*}
    \|y_j^3y_{j,k}y_{j,l}y_{k,j}^2y_{l,j}^2\|_\infty \le B
  \end{equation*}
  for all $(j,k,l) \in \{(1,2,3),(2,3,1),(3,1,2)\}$.
\end{lemma}

\section{M\"obius inversions}\label{sec:moebius}

Having found a suitable parameterization of $K$-rational points by
points over $\Od$ in an open subset of $\AA^9$ in
Lemma~\ref{lem:passage}, the main problem is essentially to estimate
the number of lattice points in the region described by the height
conditions. This is done in Lemma~\ref{lem:summations}; its proof is
defered to Section~\ref{sec:lattice}. Here, we remove the coprimality
conditions by a M\"obius inversion and recover the non-archimedian
densities.

Applying M\"obius inversion over all elements of $E'$ to the
expression in Lemma~\ref{lem:passage} gives
\begin{equation*}
  N_{U,K,H}(B)=\frac{1}{\wK^7}\sum_{\dd \in\Nd^{E'}}\prod_{\alpha\in
    E'}\mu(d_\alpha)\sum_{\substack{\yy\in(\Odnz)^V\cap M(B)\\d_{\{u,v\}}|y_u,y_v\,\forall\{u,v\}\in
      E'}}1.
\end{equation*}
We collect all terms dividing some $y_v$ to obtain
\begin{equation*}
  N_{U,K,H}(B)=\frac{1}{\wK^7}\sum_{\dd\in\Nd^{E'}}\prod_{\alpha\in E'}
  \mu(d_\alpha)\sum_{\substack{\yy\in(\Odnz)^V\cap M(B)\\r_v|y_v\,\forall v\in V}}1,
\end{equation*}
where $r_v$ is defined as the lowest common multiple of the $d_\alpha$
with $\alpha\in E'$ and $v\in\alpha$. This sum can be estimated as
follows; see Section~\ref{sec:lattice} for the proof.

\begin{lemma}\label{lem:summations}
  For $\rr\in\Nd^V$, let
  \begin{equation*}
    R_1=\prod_{v\in V}\|r_v\|_\infty, \quad 
    R_2=\prod_{\substack{j, k \in \{1,2,3\}\\j \ne k}}\|r_{j,k}\|^{2/3}_\infty
    \prod_{j \in \{1,2,3\}}\|r_j\|_\infty(\max_j\|r_j\|_\infty)^{-1/2}.
  \end{equation*}
  Then
  \begin{equation*}
    \sum_{\substack{\yy\in(\Odnz)^V\cap M(B)\\r_v|y_v\,\forall v\in V}}1
    =\frac{2^7\pi^9}{6!\dK^9}\frac B{R_1}(\log(B))^6
    +O\left(\frac B{R_2}(\log(B))^5\right).
  \end{equation*}
\end{lemma}
Combining this with Lemma~\ref{lem:passage} gives
\[N_{U,K,H}(B)=\frac{2^7\pi^9}{6!\wK^7\dK^9}\omega B(\log(B))^6+O(\rho
B(\log B)^5),\] where
\[\omega=\sum_{\dd\in\Nd^{E'}}\prod_{\alpha\in
  E'}\mu(d_\alpha)\frac{1}{R_1}, \quad
\rho=\sum_{\dd\in\Nd^{E'}}\prod_{\alpha\in
  E'}|\mu(d_\alpha)|\frac{1}{R_2}
\]
with $R_1$ and $R_2$ depending on $\dd$.

To show that $\omega$ and $\rho$ are well-defined it will enough to
show the convergence of the defining sum of $\rho$ since $|\omega| \le
\rho$. The Euler factor of $\rho$ corresponding to some prime $p \in
\Nd$ is $1+O(\|p\|_\infty^{-7/6})$. Indeed, the factor will have only
finitely many non-vanishing summands since $\mu(p^e)=0$ for all $e\ge
2$. For $\dd=(1,\dotsc,1)$, we have $R_2=1$. For $\dd$ with
$d_\alpha=p$ for at least one $\alpha=\{u,v\} \in E'$, we have
$r_u=r_v=p$ and therefore $R_2\ge\|r_u\|_\infty^{a} \|r_v\|_\infty^{b}
\ge \|p\|_\infty^{7/6}$ where $a,b \in \{\frac 1 2, \frac 2 3, 1\}$,
with at most one of them equal to $\frac 1 2$.

Let us now calculate the Euler factors $A_p$ of $\omega$ for some
prime $p \in \Nd$. Let $A\in\ZZ[x]$ be the polynomial
\[A(x)=\sum_{\tdd \in \{0,1\}^{E'}}\prod_{\alpha\in
  E'}\tmu(\td_\alpha)x^{\sum_{v\in V}\tr_v}.\] Here, $\trr \in
\{0,1\}^V$ is defined depending on $\tdd$ as follows: For any $v\in
V$, the number $\tr_v$ is the maximum of all $\td_\alpha$ with
$v\in\alpha$. Furthermore, $\tmu$ is defined by
\begin{equation*}
  \tmu(n)=\begin{cases}
    1, & n=0,\\
    -1, & n=1.
  \end{cases}
\end{equation*}
Then $A_p=A(\|p\|_\infty^{-1})$.

By further M\"obius inversions, we have
\begin{align*}
  A(x)&=\sum_{\tkk \in \{0,1\}^V} x^{\sum_{v \in V} \tk_v} \sum_{\substack{\tdd \in \{0,1\}^{E'}\\\trr = \tkk}} \prod_{\alpha \in E'} \tmu(\td_\alpha)\\
  &=\sum_{\tnn\in\{0,1\}^V}\prod_{v\in V}e(\tn_v)\sum_{\substack{\tdd\in\{0,1\}^{E'}\\\td_\alpha\le\tn_v\text{ if }v\in\alpha}}\prod_{\alpha\in E'}\tmu(\td_\alpha),
\end{align*}
where the function $e:\{0,1\}\to\QQ[x]$ defined by
\[e(n)=\begin{cases}1-x, & n=0, \\
  x, & n=1\end{cases}\]
is chosen such that
\begin{equation*}
  \sum_{k \in \{0,1\}} x^k F(k) = \sum_{n \in \{0,1\}} e(n) \sum_{0 \le s \le n} F(s)
\end{equation*}
for any function $F: \{0,1\} \to \ZZ$; this is applied above $\#V$
times. We have also used that $\tr_v\le\tn_v$ if $\td_\alpha \le
\tn_v$ for all $\alpha \in E'$ containing $v \in V$.

Note that since $\tmu(0)+\tmu(1)=0$ for a fixed $\tnn\in\{0,1\}^V$,
the sum over $\tdd$ will vanish if two vertices of $\{v\in V \mid
\tn_v=1\}$ can be joined by a line in $E'$. So it will not vanish
only if either all $\tn_v$ are $0$, exactly one of the nine $\tn_v$ is
equal to $1$ or exactly two $\tn_v$, $\tn_w$ are $1$ where $\{v,w\}$
is one of the nine edges $E$. So we have
\[A(x)=(1-x)^9+9(1-x)^8x+9(1-x)^7x^2=(1-x)^7\cdot(1+7x+x^2)\]
and finally
\[\omega=\prod_p(1-\|p\|_\infty^{-1})^7(1+7\|p\|_\infty^{-1}+\|p\|_\infty^{-2}).\]
Up to the proof of Lemma~\ref{lem:summations}, this completes the
proof of our main theorem.

\section{Estimations of lattice points}\label{sec:lattice}

In this section, we prove Lemma~\ref{lem:summations}.  We proceed as in
\cite{MR2000f:11080}. First, we rewrite the sum as
\begin{equation*}
  \sum_{\substack{\yy\in(\Odnz)^V\cap M(B)\\r_v|y_v\,\forall v\in V}}1
  =\sum_{\substack{\zz\in(\Odnz)^V\\\|z_j\|_\infty\le C_j}}1,
\end{equation*}
where we define
\begin{equation*}
  \zeta_j=z_{j,k}z_{j,l}z_{k,j}^2z_{l,j}^2,\quad C_j=B^{1/3}\|\zeta_jr_j^3r_{j,k}r_{j,l}r_{k,j}^2r_{l,j}^2\|_\infty^{-1/3}
\end{equation*}
for any $\{j,k,l\}=\{1,2,3\}$.

From here, unless stated otherwise, we use the convention that whenever $j,k$
appears in a statement, we mean all $j, k \in \{1,2,3\}$ with $j \ne k$, and
whenever $j$ shows up, we mean all $j \in \{1,2,3\}$.

To sum over $z_j$ for $j=1,2,3$, one can use the following estimate on
the number of integers in $\Odnz$ in the circle
$\Ci(C)=\{z\in\CC\mid\|x\|_\infty\le C\}$ of radius $\sqrt C$ in the complex
plane.  It seems interesting to note that
\cite[\S 4]{MR2000f:11080} finds it convenient to use the similar
estimate $C+O(\sqrt{C})$ instead of $C+O(1)$ (which is not available
in our case) for the number of natural numbers smaller than $C$.

\begin{lemma}\label{lem:circ}
  For any positive $C\in\RR$, we have
  \begin{equation*}
    \#(\{x\in\Odnz\}\cap \Ci(C))=\frac{2\pi}{\dK} C + O(\sqrt C).
  \end{equation*}
\end{lemma}
\begin{proof}
  The theorem follows in the case $C\ge C_0$ for some $C_0>0$ by the
  theorem on lattice points in homogenously expanding sets since the
  area of the circle is $\pi C$ and the area of a fundamental domain
  is $\dK/2$. The missing point in the origin can be accounted for in
  the error term in this case. If $C<1$, then
  $\Ci(C)\cap\Odnz=\emptyset$ and we have $(2\pi/\dK)C=O(\sqrt{C})$,
  so the lemma is also true in this case. Finally, if $1\le C\le C_0$,
  then $\Ci(C)\subseteq \Ci(C_0)$ and $\sqrt{C}\ge 1$. Therefore, we
  can choose the implied constant in the error term to be greater than
  $\#(\Ci(C_0)\cap\Odnz)+(2\pi/\dK)C_0$, which establishes the lemma
  in the remaining case.
\end{proof}

Therefore, with
$B_j=B\|r_{j,k}r_{j,l}r_{k,j}^2r_{l,j}^2\|_\infty^{-1}$ for all
$\{j,k,l\}=\{1,2,3\}$,
\begin{equation*}
 \begin{split}
   \sum_{\substack{\zz\in(\Odnz)^V\\\|z_j\|_\infty\le C_j}} 1
   &{}=\sum_{\substack{z_{j,k}\in\Odnz\\\|\zeta_j\|_\infty\le B_j}}
   \prod_{j=1}^3\left(\frac{2\pi}{\dK} C_j
     +O(\sqrt{C_j})\right)\\
   &{}=\sum_{\substack{z_{j,k}\in\Odnz\\\|\zeta_j\|_\infty\le B_j}}
   \left\{\frac{2^3\pi^3}{\dK^3}C_1C_2C_3 +
     O\left(C_1C_2C_3\max_j\left(C_j^{-1/2}\right)\right)\right\}\\
   &{}=\frac{2^3\pi^3B}{\dK^3R_1}\M(B,\rr)+O\left(\frac{B^{5/6}}{R_2}\R(B,\rr)\right),
  \end{split}
\end{equation*}
where
\begin{equation*}
  \M(B,\rr)=\sum_{\substack{z_{j,k}\in\Odnz\\\|\zeta_j\|_\infty\le B_j}}
  \prod_{j,k}\|z_{j,k}\|_{\infty}^{-1}
\end{equation*}
and
\begin{equation*}
\R(B,\rr)=\sum_{\substack{z_{j,k}\in\Odnz\\\|\zeta_j\|_\infty\le B_j}}
  \prod_{j,k}\|z_{j,k}\|_{\infty}^{-1}\max_j\|\zeta_j\|_\infty^{\frac 1 6}.
\end{equation*}

Let us begin with the estimation of the error term $\R$. Because of
symmetry it will be no loss to assume that
$\|\zeta_1\|_\infty\ge\|\zeta_2\|_\infty,\|\zeta_3\|_\infty$. Then
\begin{align*}
  \R(B,\rr)&{}=\sum_{\substack{z_{j,k}\in\Odnz\\\|\zeta_j\|_\infty\le B_j}}\|z_{2,1}z_{3,1}\|_\infty^{-2/3}\|z_{2,3}z_{3,2}\|_\infty^{-1}\|z_{1,2}z_{1,3}\|_\infty^{-5/6}\\
  &{}\ll\sum_{\substack{z_{j,k}\in\Odnz\,\forall j\neq1\\\|z_{j,k}\|_\infty\le B\,\forall
      j\neq1}}\|z_{2,1}z_{3,1}\|_\infty^{-2/3}\|z_{2,3}z_{3,2}\|_\infty^{-1}
  \sum_{\substack{u \in \Odnz\\\|u\|_\infty\le
      U}}\frac{d(u)}{\|u\|_\infty^{5/6}}
\end{align*}
where $U$ is defined as $U=B\|z_{2,1}z_{3,1}\|_\infty^{-2}$ and $d$ is
the divisor function in $\Odnz$.

Here, we need the following auxiliary result, which we will use again
later.

\begin{lemma}\label{lem:aux}
  For all sufficiently large $B$, we have
  \begin{equation*}
    \sum_{\substack{x \in \Odnz\\\|x\|_\infty \le B}} \|x\|_\infty^\alpha =
      \begin{cases}
        O(B^{\alpha+1}), & -1 < \alpha\le 0,\\
      O(\log B), & \alpha = -1.
      \end{cases}
  \end{equation*}
\end{lemma}

\begin{proof}
  For any $n\in\NN$, let
  \begin{equation*}
    a_n=\#\{x\in\Odnz\mid\|x\|_\infty=n\}.
  \end{equation*}
  By the Abel summation formula, we have 
  \begin{equation*}
    \sum_{\substack{x \in \Odnz\\\|x\|_\infty \le B}} \|x\|_\infty^\alpha=\sum_{1 \le n \le B} a_nn^\alpha=B^\alpha\sum_{1\le n \le B} a_n-\alpha\int_1^Bx^{\alpha-1}\sum_{1 \le n \le x} a_n\ddd x.
  \end{equation*}
  We apply Lemma~\ref{lem:circ} to the sums over $n$. The first
  term is $O(B^{\alpha+1})$. For $-1<\alpha \le 0$, the second
  term is $O(B^{\alpha+1})$ as well; for $\alpha=-1$, it is $O(\log
  B)$.
\end{proof}

Using Lemma~\ref{lem:aux} twice, the inner sum can be estimated elementarily
as
\begin{align*}
  \sum_{\|u\|_\infty\le U}\frac{d(u)}{\|u\|_\infty^{5/6}}&=\sum_{\|u\|_\infty\le U}\sum_{v\mid u}\|u\|_\infty^{-5/6}=\sum_{\|v\|_\infty\le U}\|v\|_\infty^{-5/6}\sum_{w\le U\|v\|_\infty^{-1}}\|w\|_\infty^{-5/6}\\
  &=O\left(U^{1/6}\sum_{\|v\|_\infty\le B}\|v\|_\infty^{-1}\right)=O(U^{1/6}\log B).
\end{align*}
Inserting this into the original expression for $\R(B,\rr)$ and applying
Lemma~\ref{lem:aux} again gives
\begin{align*}
  \R(B,\rr)&{}\ll B^{1/6}\log B\sum_{\substack{z_{j,k}\in\Odnz\,\forall j \neq 1\\\|z_{j,k}\|_\infty\le B\,\forall j \neq 1}}\|z_{2,1}z_{3,1}z_{2,3}z_{3,2}\|_\infty^{-1}\\
  &{}\ll B^{1/6}\log B\Bigg(\sum_{\substack{z\in\Odnz\\\|z\|_\infty\le B}}
      \|z\|_\infty^{-1}\Bigg)^4\ll B^{1/6}(\log B)^5.
\end{align*}

Now it will be enough to show that
\begin{equation*}
\M(B,\rr)=\frac{2^4\pi^6}{6!\dK^6}(\log B)^6+O(R_3(\log B)^5)
\end{equation*}
where we define $R_3=\prod_{j,k}\|r_{j,k}\|_\infty^{1/3}$. Assume this
is done for the case of $r_{j,k}=1$ for all $j,k$ and all $B\ge B_0$
for some $B_0>1$. Then on the one hand
\[\M(B,\rr)\le\M(B,(1,\dotsc,1))
=\frac{2^4\pi^6}{6!\dK^6}(\log B)^6+O((\log B)^5)\] and on the other
hand
\[\M(B,\rr)\ge\M(B/R_3^6,(1,\dotsc,1))=\frac{2^4\pi^6}{6!\dK^6} 
(\log(B/R_3^6))^6+O((\log(B/R_3^6))^5)\] for all $\rr$ with
$R_3\le(B/B_0)^{1/6}$. This gives the required estimate in this case
since there is a constant $C$ such that $\log R_3 \le CR_3^{1/6}$ for
any $R_3\ge 1$. Otherwise we notice that the error term dominates the
main term.

It therefore remains to estimate $\M(B)=\M(B,(1,\dotsc,1))$. In this
case, $B_j=B$.

\begin{lemma}\label{lem:sum_integral}
  Let
  \begin{equation*}
    N(B)=\{\zz\in\CC^6\mid\|z_{j,k}\|_\infty\ge 1, \|\zeta_j\|_\infty\le B\},
  \end{equation*}
    where the $\zeta_j$ are defined as before. Define the integral
  \begin{equation*}
    \I(B)=\left(\frac 2 \dK\right)^6\int_{N(B)}\prod_{j,k}\frac{\ddd z_{j,k}}{\|z_{j,k}\|_\infty}.
  \end{equation*}
  Then $\M(B)=\I(B)+O((\log B)^5)$ for all sufficiently large
  $B$.
\end{lemma}

\begin{proof}
  We fix a fundamental domain $F$ of the lattice corresponding to
  $\Od$ in $\CC$; its area is $\dK/2$. Our goal is to compare the
  terms $\prod_{j,k}\|z_{j,k}\|_\infty^{-1}$ of the sum defining
  $\M(B)$ with integrals over translations of $F$. For an upper bound
  for $\M(B)$ compared to $\I(B)$, we must choose a translation $F(z)$
  of $F$ whose elements are closer to $0$ than $z \in \Od$. For an
  upper bound for $\I(B)$ compared to $\M(B)$, we must choose a
  translation $F'(z)$ of $F$ whose elements are further away from $0$
  than $z \in \Od$. Furthermore, we must be careful to stay away from
  the ball $\|z\|_\infty \le 1$ and from the real and imaginary axes.

  Let $R$ be the smallest rectangle whose sides (of real length $l_r$
  resp.\ imaginary length $l_i$) are parallel to the real and
  imaginary axes and that contains $F$. For any $z \in \CC$ with real
  part $|\Re(z)| \ge 1+l_r$ (resp.\ $|\Re(z)| \ge 1$) and imaginary
  part $|\Im(z)| \ge 1+l_i$ (resp.\ $|\Im(z)| \ge 1$), let $R(z)$
  (resp.\ $R'(z)$) be the unique translation of the rectangle $R$ with
  the following property: The point $z \in \CC$ is the corner with the
  largest (resp.\ smallest) distance to $0 \in \CC$ of $R(z)$ (resp.\
  $R'(z)$). Let $F(z)$ (resp.\ $F'(z)$) be the unique translation of
  $F$ contained in $R(z)$ (resp.\ $R'(z)$). For any $x \in F(z)$
  (resp.\ $x \in F'(z)$), we have $\|z\|_\infty \ge \|x\|_\infty$
  (resp.\ $\|z\|_\infty \le \|x\|_\infty$).

  Let
  \begin{equation*}
    E(B) = \{\zz \in N(B) \mid 
    |\Re(z_{j,k})| \ge 1+l_r, |\Im(z_{j,k})| \ge 1+l_i\}
  \end{equation*}
  and $G(B) = N(B) \setminus E(B)$. Let
  \begin{equation*}
    G'(B) = \{\zz \in \CC^6 \mid 
    1 \le \|z_{j,k}\|_\infty \le B, |\Re(z_{1,2})| \le 1+l_r\}.
  \end{equation*}
  We note that $G(B)$ is contained in the union of $G'(B)$ with eleven
  other sets of a similar shape (with the analogous condition on
  $\Re(z_{j,k})$ or $\Im(z_{j,k})$) that we will be able to deal with
  in the same way as $G'(B)$.

  First, we give an upper bound for $\M(B)$ in terms of $\I(B)$.  We
  split $\M(B)$ into a sum over $E(B) \cap (\Odnz)^6$ giving the main
  term and a sum over $G(B) \cap (\Odnz)^6$ giving the error term. For
  the main term, we note that the sets $\prod_{j,k}F(z_{j,k})$ for all
  $\zz \in E(B) \cap (\Odnz)^6$ are subsets of $N(B)$ whose pairwise
  intersections are null sets. As $\|z_{j,k}\|_\infty \ge
  \|x\|_\infty$ for any $x \in F(z_{j,k})$ with $\zz \in E(B) \cap
  (\Odnz)^6$, we have $\|z_{j,k}\|_\infty^{-1} \le
  \frac{2}{\dK}\int_{F(z_{j,k})} \|x\|_\infty^{-1}\ddd x$. Therefore,
  \begin{equation*}
    \sum_{\zz \in E(B) \cap (\Odnz)^6} \prod_{j,k} \|z_{j,k}\|_\infty^{-1} 
    \le \sum_{\zz \in E(B) \cap (\Odnz)^6} \prod_{j,k} \frac{2}{\dK} 
    \int_{F(z_{j,k})} \frac{\ddd x_{j,k}}{\|x_{j,k}\|} 
    \le \I(B).
  \end{equation*}
  For the error term, we deal with $G'(B)$ instead of $G(B)$, as
  mentioned before; here,
  \begin{equation*}
    \sum_{\zz \in G'(B) \cap (\Odnz)^6} \prod_{j,k} \|z_{j,k}\|_\infty^{-1} \ll (\log B)^5\sum_{\substack{z_{1,2} \in \Odnz\\|\Re(z)| \le 1+l_r}} \|z_{1,2}\|_\infty^{-1} \ll (\log B)^5.
  \end{equation*}
  Indeed, in the first step, we use Lemma~\ref{lem:aux}. In the second
  step, let $N$ be the maximum number of lattice points in $\Od$ in a
  box of real length $1+l_r$ and imaginary length $1$. We note that
  the sum over $z_{1,2}$ is bounded because all $z_{1,2} \in \Odnz$
  with $|\Re(z_{1,2})| \le 1+l_r$ and $|\Im(z_{1,2})| \le 1$
  contribute $\le 4N$, and all $z_{1,2} \in \Od$ with $k \le
  |\Im(z_{1,2})| \le k+1$ contribute $\le 4N k^{-2}$ (because
  $\|z_{1,2}\|_\infty \ge k^2$), which converges when summed over $k
  \in \NN$. In total,
  \begin{equation*}
    \M(B) \le \I(B) + O((\log B)^5).
  \end{equation*}

  For the other direction, we note that, for any $x$ with $|\Re(x)|
  \ge 1+l_r$ and $|\Im(x)| \ge 1+l_i$, there is a $z \in \Od$ such
  that $x \in F'(z)$, with $|\Re(z)| \ge 1$ and $|\Im(z)| \ge 1$ and
  $\|z\|_\infty \le \|x\|_\infty$, by our construction of
  $F'(z)$. Therefore, for any $\xx \in E(B)$, there is a $\zz \in
  N(B) \cap (\Odnz)^6$ such that $\xx \in
  \prod_{j,k}F'(z_{j,k})$. Thus $E(B)$ is covered by $\bigcup_{\zz \in
    N(B) \cap (\Odnz)^6} \prod_{j,k} F'(z_{j,k})$, and
  \begin{equation*}
    \left(\frac 2\dK\right)^6\int_{E(B)} \prod_{j,k} \frac{\ddd x_{j,k}}{\|x_{j,k}\|_\infty} 
    \le \sum_{\zz \in
      N(B) \cap (\Odnz)^6} \prod_{j,k} \frac{2}{\dK}\int_{F(z_{j,k})} \frac{\ddd x_{j,k}}{\|x_{j,k}\|_\infty} \le \M(B).
  \end{equation*}
  It remains to consider the integral over $G(B)$. Again, we just
  consider $G'(B)$. Here, we have
  \begin{align*}
    \int_{G'(B)} \prod_{j,k} \frac{\ddd x_{j,k}}{\|x_{j,k}\|_\infty} &\ll (\log B)^5 \int_{1\le\|x_{1,2}\|_\infty\le B,\ |\Re(x_{1,2})|\le 1+l_r}\frac{\ddd x_{1,2}}{\|x_{1,2}\|_\infty}\\
    &\ll (\log B)^5\left(1 + 2(1+l_r) \int_1^\infty \frac{\ddd x}{x^2}\right)
      \ll (\log B)^5
  \end{align*}
  because of $\int_{1 \le \|x_{j,k}\|_\infty \le B}
  \|x_{j,k}\|_\infty^{-1} \ddd x_{j,k} \ll \log B$ and
  $\|x_{1,2}\|_\infty \ge |\Im(x_{1,2})|^2$ and the boundedness of the
  integral over all $x_{1,2}$ as above with $|\Im(x_{1,2})| \le 1$.
  Therefore,
  \begin{equation*}
    \I(B) \le \M(B)+O((\log B)^5),
  \end{equation*}
  completing the proof.
\end{proof}

It remains to evaluate $\I(B)$. Using the rotation symmetries of its
integrands, one can write
\begin{align*}
  \I(B)=\left(\frac 2 \dK\right)^6 \pi^6\int\frac{\ddd
    z_{1,2}}{z_{1,2}}\dotsm\frac{\ddd z_{3,2}}{z_{3,2}},
\end{align*}
where the integral runs now over all real $z_{j,k}\ge 1$ satisfying the three
inequalities $z_{j,k}z_{j,l}z_{k,j}^2z_{l,j}^2\le B$ for
$\{j,k,l\}=\{1,2,3\}$. Substituting $z_{j,k}=B^{t_{j,k}}$ shows that
\begin{equation*}
  \I(B)=V\frac{2^6\pi^6}{\dK^6}(\log B)^6,
\end{equation*}
where $V$ denotes the integral
\[V=\int\ddd t_{1,2}\dotsm\ddd t_{3,2}\] over the six-dimensional convex
polytope defined by the six inequalities $t_{j,k}\ge 0$ and the three
inequalities
\begin{equation*}
  t_{j,k}+t_{j,l}+2t_{k,j}+2t_{l,j} \le 1
\end{equation*}
for all $\{j,k,l\}=\{1,2,3\}$. The volume of this polytope is
$V=(4\cdot 6!)^{-1}$ \cite{MR2000f:11080}. This completes the proof of
Lemma~\ref{lem:summations}.

\bibliographystyle{alpha}

\bibliography{manin_dp3_3a2_gauss}

\end{document}